\documentclass[a4paper,12pt]{amsart}

\usepackage{amssymb,amsmath,comment}
\usepackage{array}
\usepackage{longtable}
\usepackage{amsfonts}
\usepackage{latexsym}
\usepackage{amssymb}
\usepackage{amsthm}
\usepackage{dsfont}
\usepackage{tikz, subfigure}
\usepackage{verbatim}
\usepackage{graphicx}
\usepackage{float}
\restylefloat{table}
\numberwithin{table}{section}
\numberwithin{figure}{section}
\numberwithin{equation}{section}

\newtheorem{theorem}{Theorem}[section]

\newtheorem{lemma}[theorem]{Lemma}
\newtheorem{cor}[theorem]{Corollary}
\newtheorem{remark}[theorem]{Remark}

\newtheorem{conj}[theorem]{Conjecture}

\newcommand{\X}{\mathcal{X}}
\newcommand{\M}{\mathcal{M}}

\begin{document}
\title
{Conjugates of Pisot numbers \\}
\date{\today}
\author[Kevin G. Hare]{Kevin G. Hare}
\address{Department of Pure Mathematics, University of Waterloo, Waterloo, Ontario, Canada N2L 3G1}
\email{kghare@uwaterloo.ca}
\thanks{Research of K.G. Hare was supported by NSERC Grant 2019-03930}
\author[Nikita Sidorov]{Nikita Sidorov}
\address{Department of Mathematics, The University of Manchester, Manchester, M13 9PL, United Kingdom}
\address{and}
\address{Department of Pure Mathematics, University of Waterloo, Waterloo, Ontario, Canada N2L 3G1}
\email{sidorov@manchester.ac.uk}
\thanks{Research of N. Sidorov was supported in part by University of Waterloo}

\begin{abstract}
In this paper we investigate the Galois conjugates of a Pisot number $q \in (m, m+1)$, $m \geq 1$.
In particular, we conjecture that
    for $q \in (1,2)$ we have $|q'| \geq \frac{\sqrt{5}-1}{2}$ for all
    conjugates $q'$ of $q$.
Further, for $m \geq 3$, we conjecture that
    for all Pisot numbers $q \in (m, m+1)$
    we have $|q'| \geq \frac{m+1-\sqrt{m^2+2m-3}}{2}$.
A similar conjecture if made for $m =2$.
We conjecture that all of these bounds are tight.
We provide partial supporting evidence for this conjecture.
This evidence is both of a theoretical and computational nature.

Lastly, we connect this conjecture to a result on the dimension of Bernoulli
    convolutions parameterized by $\beta$,
    whose conjugate is the reciprocal of a Pisot number.
\end{abstract}


\maketitle

\section{Introduction}
In this paper we investigate the conjugates of a Pisot number $q \in (m, m+1)$.
In particular, we conjecture a tight value for $c_m > \frac{1}{m+1}$ such
    such that for all Pisot numbers $q \in (m, m+1)$ we have that
    all conjugates $|q'| \geq c_m$.
Precise values for a tight lower bound are given in Conjecture~\ref{conj:exact}.
Partial theoretical supporting evidence for these conjectures is given
    in Theorems~\ref{thm:complex}, \ref{thm:non-unit}, \ref{thm:parry}
    and \ref{thm:m=1}.
Computational supporting evidence is given in
    Theorems~\ref{thm:1.933}, \ref{thm:150}
    and Table~\ref{tab:counter}.

Recall that a {\em Pisot number} is a positive real algebaric integer all of whose other conjugates
    are strictly less than 1 in absolute value.
We say that a real number $\theta > 1$ is a {\em Parry number} if its gredy $\theta$-expansion is
    eventually periodic.
That is, if the orbit of $1$ under the transformation $x \mapsto x \theta \mod 1$ is finite.
It is clear that such $\theta$ are algebraic integers.
Every Pisot number is known to be a Parry number, see
    \cite{Bertrand77, Schmidt80}.
In \cite{So94} Solomyak looked at the conjugates of Parry numbers, and showed
    that they are always strictly less than $\frac{1+\sqrt{5}}{2}$, the golden
    ratio.
He also analyzed the domain of all such conjugates.

Let $P(x) = a (x-\alpha_1) \dots (x-\alpha_d) \in \mathbb{Z}[x]$.
Recall that the Mahler measure of a polynomial is defined as
    \[ M(P) = M(a (x-\alpha_1) \dots (x-\alpha_d)) = |a| \prod \max(1, |\alpha_i|). \]
We define the Mahler measure of an algebraic number as the Mahler measure of it's minimal
    polynomial.
It is clear that the Mahler measure of a Pisot number is the Pisot number.

In \cite{Du11}, Dubickas considered integer polynomials $P$ with no reciprocal
    factor and non-zero constant term.
In this paper they gave a lower bound for the roots of $P$ in terms of the
    Mahler measure.
In particular, it was shown that the conjugates were bounded below by
   \[ M(P)^{-1} + 2^{-2k-7} M(P)^{-3} (M(P)^2 \log(M(P)))^{-(k+1)/(d-k-1)} \]
where $P$ is of degree $d$ and $1 \leq k \leq d-2$ is the number of roots outside the unit circle.
In the special case where $P$ is the minimal polynomial of a Pisot number $q$ of degree $d \geq 4$,
    a tighter bound is given.
The conjugates $q'$ of the Pisot number satisfy
    \[ |q'| \geq q^{-1} +2^{-5} q^{-3-4/(d-2)} (\log(q))^{-2/(d-2)} \]
In the special case where $q$ is a cubic Pisot number, it is shown that
    $|q'| \geq q^{-1} +1.999 q^{-2}$.

Consider a Pisot number in $(m, m+1)$.
All of its conjugates satisfy
    $|q'| \geq 1/(m+1)$ as the product of the conjugates is the norm
    of the Pisot number, and is a non-zero integer.
From \cite{Du11} above, we have a slightly stronger bound, namely that for
    for $q \in (m, m+1)$ and $d \geq 4$ that $|q'| \geq \frac{1}{m+1} + \frac{1}{32 (m+1)^3}$.
Similarly, for $q \in (m, m+1)$ and $d = 3$ we have
    $|q'| \geq \frac{1}{m+1} +\frac{1.999}{(m+1)^{2}}$.

Salem proved that the set of Pisot numbers is closed \cite{Cassels}.
It is easy to show that this is not the case for the conjugates of Pisot
    numbers.
(In particular $1$ is a limit point of the conjugates of Pisot numbers, but
    $1$ is not a conjugate of a Pisot number.)
An interesting question that we investigate is, what the limit points of the
    conjugates of Pisot numbers look like.
For example, we know that for $q \in (m, m+1)$ the conjugates
    are strictly greater than $1/(m+1)$ in absolute value (and in fact bounded away
    from Dubickas).
Can we find a stronger bound for the conjugates?
Is there structure to the limit points of the conjugates?
See Figure~\ref{fig:conj} and \ref{fig:conjF} for the conjugates of Pisot numbers $q \in (1,2)$ for degree at most 30 and 35 respectively.

\begin{figure}
\centering \scalebox{0.35}
{\includegraphics{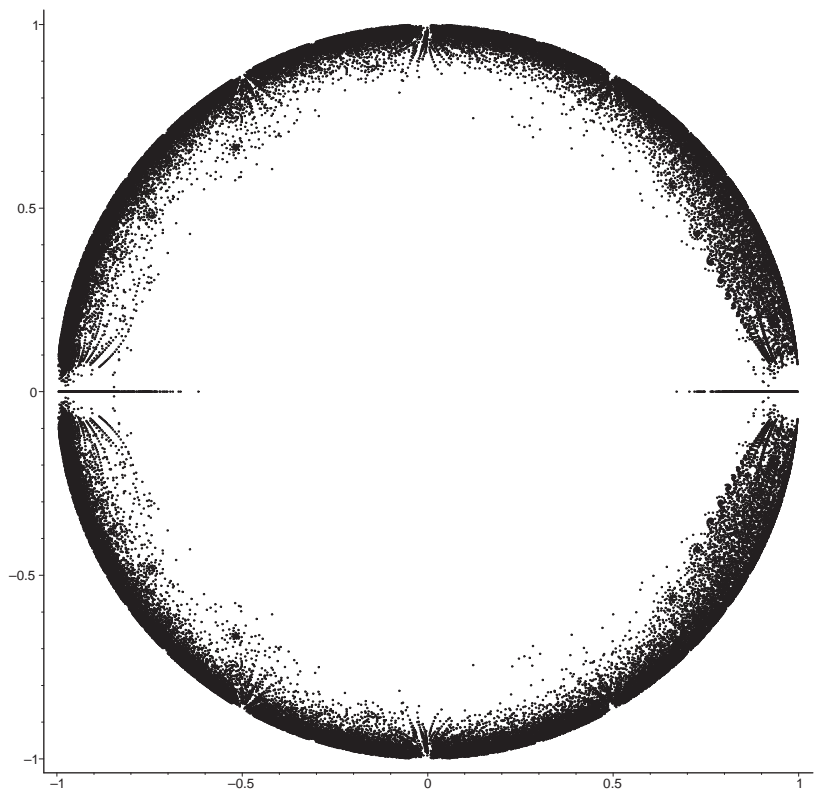}}
\caption{Conjugates of Pisot numbers $q \in (1,2)$ of degree at most 30}
\label{fig:conj}
\end{figure}

\begin{figure}
\centering \scalebox{0.35}
{\includegraphics{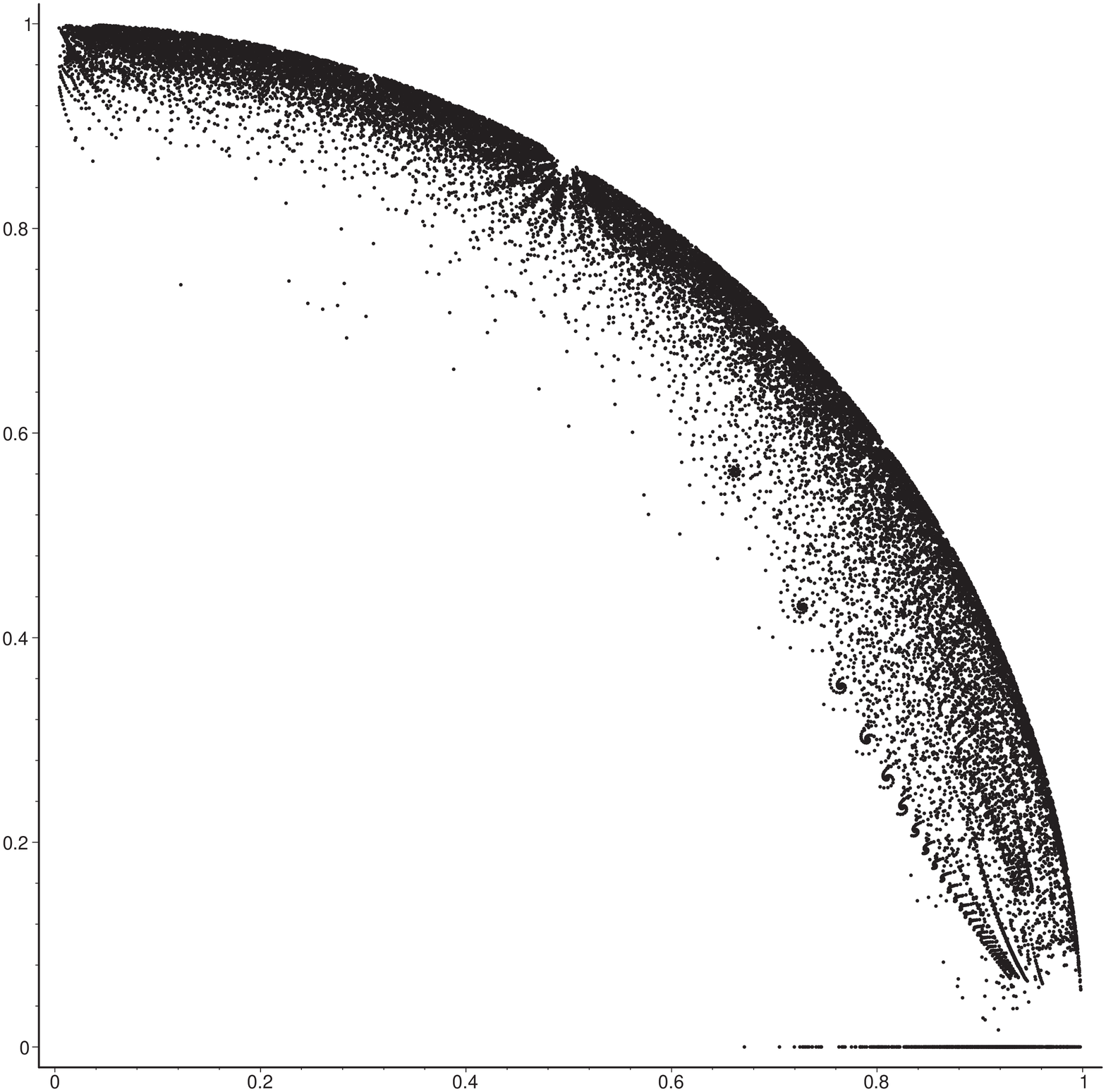}}
\caption{Conjugates of Pisot numbers $q \in (1,2)$ of degree at most 35
    in the first quadrant}
\label{fig:conjF}
\end{figure}

\begin{conj}
\label{conj:exact}
\begin{itemize}
\item
For all Pisot numbers $q \in (1,2)$, we have that
    $|q'| \geq \frac{\sqrt{5}-1}{2}$ for all conjugates $q'$ of $q$.
\item
For all Pisot numbers $q \in (2,3)$, we have that
    $|q'| \geq c_2$ for all conjugates $q'$ of $q$.
Here $c_2$ is the absolute value of the root
    of minimal modulus of $x^4-3 x^3+x^2-2x-1$.
\item
For all $m \geq 3$ and
    all Pisot numbers $q \in (m,m+1)$, we have that
    $|q'| \geq \frac{m+1-\sqrt{m^2+2m-3}}{2}$ for all conjugates $q'$ of $q$.
\end{itemize}
\end{conj}

For ease of notation,
\begin{enumerate}
\item Set $c_1 = \frac{\sqrt{5}-1}{2}$,
    the absolute value of the root of minimal modulus of $x^2-x-1$,
\item Set
    $c_2$ as the absolute value of the root
    of minimal modulus of $x^4-3 x^3+x^2-2x-1$, and
\item For $m \geq 3$ set
    $c_m = \frac{m+1-\sqrt{m^2+2m-3}}{2}$ is the root of minimal
    modulus of $x^2-(m+1) x +1$, for $m \geq 3$.
\end{enumerate}

It is easy to observe that all $c_m$, $m \geq 1$ are absolute values of
    conjugates of Pisot numbers.
It should be noted that the bound given in Conjecture \ref{conj:exact} on the location of the conjugates
    is significantly stronger than that known bound given by \cite{Du11}.

Amara gave a complete description of the limit points of Pisot numbers in $(1,2)$ in
    \cite{Amara66} (see Theorem~\ref{Amara}).
We will denote these limit points as $\phi_r,\psi_r$ and $\chi$, where $r = 1, 2, 3, \dots$.
Furthermore, a description of the sequence of Pisot numbers approaching
    each of $\phi_r, \psi_r$ or $\chi$ was also given.
The Pisot numbers in these sequences, together with there limit points,
    are called {\em regular Pisot numbers}.
It was shown that the only Pisot numbers sufficiently close to these limit points are
    regular Pisot numbers.
Furthermore, for any $\varepsilon > 0$, Amara showed that there are only a
finite number of Pisot numbers in $(1, 2-\varepsilon)$, that are not in one
of these sequences.
These are called {\em irregular Pisot numbers}.
Unfortunatley, although there are only a finite number of irregular Pisot
    numbers in $(1, 2-\varepsilon)$ for any $\varepsilon > 0$,
    this is not known to be true for the interval $(1,2)$.
This is the main issue with proving Conjecture~\ref{conj:exact} in full
    generality for $q \in (1,2)$.

In Section~\ref{sec:general} we provide partial results to this conjecture for general $m$.
Namely, we show that if $q \in (m, m+1)$ is a Pisot number and either $q$ or its conjugate
    $q'$ satisfying certain properties, then $q$ satisfies Conjecture~\ref{conj:exact}.
Sections~\ref{sec:regular 1} and \ref{sec:regular 2} consider the special case of regular Pisot numbers.
In particular we show that Conjecture~\ref{conj:exact} holds for regular Pisot numbers $q \in (1,2)$.
In addition, in Section~\ref{sec:regular 2} we  further investigate the limiting behaviour of the conjugates of regular Pisot numbers.

Computational evidence for Conjecture~\ref{conj:exact} is provided in Section~\ref{sec:m}.

This investigation was motivated by Conjecture~3.5 of \cite{HKPS},
    as it relates to the dimension of Bernoulli convolutions for a special
    family of algebraic integers.
We will discuss this more in Section~\ref{sec:Garsia}.

\section{General results}
\label{sec:general}

\begin{theorem}
\label{thm:complex}
Let $q \in (m,m+1)$ be a Pisot number.
There are no non-real complex conjugates $|q'| < \frac{1}{m+1/2}$.
There is at most one real conjugate $|q'| < \frac{1}{m+1/2}$.
\end{theorem}

\begin{proof}
We see that non-real conjugates come in complex conjugate pairs.
Hence for a non-real complex number we have
    $|q'| > 1/\sqrt{m+1} > \frac{1}{m+1/2}$.

The second observation to make is that we can have at most one real conjugate less than $1/\sqrt{m+1} > \frac{1}{m+1/2}$.
If we were to have more than this, then the constant term of the minimal
    polynomial would be less than 1 in absolute value.
\end{proof}

\begin {lemma}
We have $c_m < \frac{1}{m+1/2}$ for all $m$.
\end{lemma}

\begin{proof}
For $m=1,2$ this is a direct check. For $m\ge3$, the required inequality is equivalent to $2m+1<m+1+\sqrt{m^2+2m-3}$, which holds trivially.
\end{proof}

\begin{theorem}
\label{thm:non-unit}
Conjecture~\ref{conj:exact} is true for all non-unit Pisot numbers.
More precisely, if $q'$ is a conjugate of a non-unit Pisot number $q \in (m, m+1)$,
    then $|q'| \geq \frac{2}{m+1} \geq c_m$.
\end{theorem}

\begin{proof}
Let $q'$ be a conjugate of $q$.
As $q$ is a non-unit, we notice that
    $|q' q| \geq 2$.
This implies that $|q'| \geq 2/q \geq 2/(m+1) \geq \frac{1}{m+1/2} \geq c_m$.
Hence Conjecture~\ref{conj:exact} holds.
\end{proof}

Recall that $q$ is a {\em simple Parry number} if the greedy-expansion of $1$ is finite.

\begin{theorem}
\label{thm:parry}
Conjecture~\ref{conj:exact} is true for all Pisot numbers that are also
    simple Parry numbers.
More precisely, if $q'$ is a real Galois conjugate of a simple Parry number and Pisot number $q \in (m, m+1)$, then
\[
|q'|\ge \frac{\sqrt{m^2+4}-m}2 \ge c_m.
\]
\end{theorem}

\begin{proof}
The dominant root of $x^2-(m+1)x+1$ is clearly larger than that of $x^2-mx-1$
    if $m\geq 3$, and the second inequality can be explicitly checked for $m = 1, 2$.
Hence, it suffices to show the first inequality only.
Without loss of generality, we may assume that $|q'|<1$. Let
\[
1=a_1/q+a_2/q^2+\dots+a_n/q^n
\]
be the greedy expansion of 1.
Notice that $0\le a_j\le m$. Now, $q'$ satisfies the same
equation, which implies $q'<0$. Put $\beta=-1/q'$. Then, for $n$ odd
\begin{align*}
\beta^n
& =a_{n-1}\beta^{n-1}-a_{n-2}\beta^{n-2}+\dots \\
& \le m(\beta^{n-1}+\beta^{n-3}+\dots+1) \\
& = m \frac{\beta^{n+1}-1}{\beta^2-1} < \beta^n
\end{align*}
if $\beta>(\sqrt{m^2+4}+m)/2$, a contradiction.
A similar result holds when $n$ is even.
\end{proof}

\begin{remark}
It is worth noting that the above result is tight by considering the
    the root of $x^2-mx -1$, which is both a Pisot number and a simple
    Parry number.
\end{remark}

\begin{remark}
Notice, this result hold for real conjugates of non-Pisot simply Parry numbers as well.
\end{remark}

\section{Regular Pisot numbers in $(1,2)$, part 1}
\label{sec:regular 1}

Amara \cite{Amara66} gave a complete description of regular Pisot numbers, given below.

\begin{theorem}\label{Amara}
The limit points of the set of Pisot numbers in $(1,2)$ are the following:
$$\phi_1=\psi_1<\phi_2<\psi_2<\phi_3<\chi
<\psi_3<\phi_4< \cdots <\psi_r<\phi_{r+1}< \cdots <2$$
where
$$
\begin{cases}
\text{the minimal polynomial of\ } \phi_r \text{\ is\ } \Phi_r(x) := x^{r+1}-2x^r+x-1, \\
\text{the minimal polynomial of\ } \psi_r \text{\ is\ } \Psi_r(x) := x^{r+1}-x^r-\cdots-x-1, \\
\text{the minimal polynomial of\ } \chi \text{\ is\ } \X(x) := x^4-x^3-2x^2+1. \\
\end{cases}
$$
For each of these limit points, there exists an $\varepsilon$ such that all
   Pisot numbers within $\varepsilon$ of this limit point is of the form specified
   in Table~\ref{tab:regular}.

\begin{table}
{\begin{tabular}{@{}ll@{}}
Limit Point & Defining polynomials \\
\hline
$\phi_r$     & $\Phi_{A,r,n}^{\pm}(x) := \Phi_r(x) x^n\pm (x^r-x^{r-1}+1)$ \\
             & $\Phi_{B,r,n}^{\pm}(x) := \Phi_r(x) x^n\pm (x^r-x+1)$ \\
             & $\Phi_{C,r,n}^{\pm}(x) := \Phi_r(x) x^n\pm (x^r+1)(x-1)$ \\
$\psi_r$     & $\Psi_{A,r,n}^{\pm}(x) := \Psi_r(x) x^n\pm(x^{r+1}-1)$\\
             & $\Psi_{B,r,n}^{\pm}(x) := \Psi_r(x) x^n\pm(x^{r}-1)/(x-1)$ \\
$\chi$       & $\X_{A,n}^{\pm}(x) := \X(x) x^n\pm(x^3+x^2-x-1)$ \\
             & $\X_{B,n}^{\pm}(x) := \X(x) x^n\pm(x^4-x^2+1)$  \\
\end{tabular}
}
\caption{Regular Pisot numbers}
\label{tab:regular}
\end{table}
\end{theorem}

The first few limit points are:
\begin{itemize}
\item $\phi_1 = \psi_1 \approx 1.618033989$, the root in $(1,2)$ of $x^2-x-1$
\item $\phi_2 \approx 1.754877666$, the root in $(1,2)$ of $x^3-2 x^2+x-1$
\item $\psi_2 \approx 1.839286755$, the root in $(1,2)$ of $x^3-x^2-x-1$
\item $\phi_3 \approx 1.866760399$, the root in $(1,2)$ of $x^4-2 x^3+x-1$
\item $\chi \approx 1.905166168$, the root in $(1,2)$ of $x^4-x^3-2 x^2+1$
\item $\psi_3 \approx 1.927561975$, the root in $(1,2)$ of $x^4-x^3-x^2-x-1$
\item $\phi_4 \approx 1.933184982$, the root in $(1,2)$ of $x^5-2 x^4+x-1$
\end{itemize}

Recall that $c_1 = \frac{\sqrt{5}-1}{2}$.
In this section we prove that $c_1$ is a tight lower bound for all regular
    Pisot numbers in $(1,2)$.
Assuming it is a lower bound, it is easy to see that it is tight by considering
    the polynomial $x^2-x-1$.

We see that regular Pisot numbers are roots of polynomials of the form $f(x) x^n + g(x)$.
We begin this section with an investigation of polynomials of this form.

\begin{lemma}
\label{lem:suff large}
For fixed polynomials $f$ and $g$,
let $P_n(x) = f(x) x^n + g(x)$ where $g(c_1) g(-c_1) > 0$.
There exists an $n_0$, dependent on polynomials $f$ and $g$, such that $P_n(c_1) P_n(-c_1) > 0$ for all $n \geq n_0$.
\end{lemma}

\begin{proof}
This follows from the fact that $c_1^n f(c_1) \to 0$ and
    $(-c_1)^n f(-c_1) \to 0$ as $n \to \infty$.
This $n_0$ can be explicitly calculated.
\end{proof}

We say an integer polynomial $g$ is {\em of height} $h$ if all
    of its coefficients are bounded by $h$ in absolute value and at least
    one coefficient is $h$ in absolute value.

\begin{lemma}
\label{lem:suff large 2}
For fixed polynomial $g$ and polynomial $f$ of height bounded by $h$,
    let $P_n(x) = f(x) x^n + g(x)$ where $g(c_1) g(-c_1) > 0$.
There exists an $n_0$, dependent on polynoimal $g$ and height $h$, such that $P_n(c_1) P_n(-c_1) > 0$ for all $n \geq n_0$.
\end{lemma}

\begin{proof}
We notice that $|f(c_1)|$ and $|f(-c_1)|$ are bounded by a constant dependent only on $h$.
The rest of the proof is as before.
\end{proof}

\begin{theorem}
\label{thm:m=1}
Let $q \in (1,2)$ be a regular Pisot number, and $q'$ a conjugate of $q$.
Then $|q'| \geq c_1$.
\end{theorem}

\begin{proof}
Many of the regular Pisot numbers $q \in (1,2)$ can be proven to satisfy Conjecture~\ref{conj:exact} by use of Theorem~\ref{thm:complex}, Lemma~\ref{lem:suff large} and Lemma~\ref{lem:suff large 2}.
For example, consider $\Phi_r(x) = x^r(x-1)+(x-1)$.
Letting $f(x) = x-1$ and $g(x) = x-1$ we see that $\Phi_r(x) = f(x) x^r + g(x)$ meets the conditions of Lemma \ref{lem:suff large}.
In this case we can take $n_0 = 3$ and hence for all $r \geq n_0 = 3$ we have that all conjugates of $\Phi_r(x)$ satisfy $|q'| \geq c_1$.
This is summarized in the first line of Table~\ref{tab:regular proof}.
Similar informaiton is provided for other regular Pisot numbers in Table~\ref{tab:regular proof}.
\begin{table}
\begin{tabular}{lll}
Defining polynomials  & Restrictions & Note\\
\hline
$\Phi_r$ &  $r \geq 3$
    & Lemma~\ref{lem:suff large}\\
$\Psi_r$ &  $r \geq 1$
    & Multiple by $(x-1)$, Lemma~\ref{lem:suff large}\\
$\Phi_{A,r,n}^{\pm}$   & $n \geq r+1, r \geq 4$
    & Lemma~\ref{lem:suff large 2} with $h = 2$ \\
$\Phi_{B,r,n}^{\pm}$   & $n \geq r+1, r\geq 4$
    & Lemma~\ref{lem:suff large 2} with $h = 2$ \\
$\Phi_{B,r,n}^{+}$     & $n \leq r, n\geq 4$
    & Lemma~\ref{lem:suff large 2} with $h = 2$ \\
$\Phi_{C,r,n}^{\pm}$   & $n \geq r+1, r\geq 6$
    & Lemma~\ref{lem:suff large 2} with $h = 2$ \\
$\Psi_{A,r,n}^{\pm}$   & $ n \geq r+2, r \geq 2$
    & Lemma~\ref{lem:suff large 2} with $h = 1$ \\
$\Psi_{A,r,n}^{+} $    & $ n \leq r+1, n \geq 2$
    & Lemma~\ref{lem:suff large 2} with $h = 1$ \\
$\Psi_{B,r,n}^{\pm}$   & $ n \geq r, r \geq 4$
    & Multiply by $(x-1)$, Lemma~\ref{lem:suff large} \\
$\X_{A,n}^{\pm}$       & $n \geq 2$
    & Lemma~\ref{lem:suff large} \\
$\X_{B,n}^{\pm}$       & $n \geq 1$
    & Lemma~\ref{lem:suff large}
\end{tabular}
\caption{Proof for some regular Pisot}
\label{tab:regular proof}
\end{table}

There are a number of special cases that we still need to look at.
In particular, if $n < r$ or $r$ being reasonably small.

We notice that
$\Phi_{A,r,n}^{+} = \Phi_{A,n+1,r-1}^{+}$,
$\Phi_{C,r,n}^{+} = \Phi_{C,n,r}^{+}$, and
$\Psi_{B,r,n}^{+} = \Psi_{B,n,r}^{+}$,

We further notice that
$\Phi_{A,r,n}^{-}$ does not have a Pisot root in $(1,2)$ for $n < r$.
To see this, notice that $\Phi_{A,r,n}^{-}(1) = -2$ and
                          $\Phi_{A,r,n}^{-}(2) = 2^n - 2^{r-1}-1$.
Hence by the intermediate value theorem, $\Phi_{A,r,n}^{-1}$
    cannot have a single root between $1$ and $2$ when $n < r$,
    and hence cannot have a Pisot root in $(1,2)$.
This can similarly be said for
$\Phi_{B,r,n}^{-}$ and $\Psi_{B,r,n}^{-}$ for $r < n$ and for
$\Phi_{C,r,n}^{-}$ and $\Psi_{A,r,n}^{-}$ for $r \leq n$.

All remaining cases have the property that $r, n \leq 6$.
We verify that either the polynomial does not have a Pisot root in $(1,2)$,
    or that the conjecture holds for these cases.
\end{proof}

\section{Regular Pisot numbers in $(1,2)$, part 2}
\label{sec:regular 2}

As before, we note that regular Pisot numbers are roots of polynomials of the form $f(x) x^n + g(x)$.
The limit points of roots of such polynomials relate to the roots of $f$ and $g$.

\begin{theorem}
\label{thm:rouche}
\begin{enumerate}
\item Let $f(q) = 0$ for $|q| > 1$.  Then there exists a root $q_n$ of
   $f(x) x^n + g(x)$ such that $q_n \to q$ as $n \to \infty$.
\label{limits:1}
\item Let $q_n$ be roots of $f(x) x^n + g(x)$ such that $q_n \to q$
    with $|q| > 1$. Then $f(q) = 0$.
\label{limits:2}
\item Let $g(q) = 0$ for $|q| < 1$.  Then there exists a root $q_n$ of
   $f(x) x^n + g(x)$ such that $q_n \to q$ as $n \to \infty$.
\label{limits:3}
\item Let $q_n$ be roots of $f(x) x^n + g(x)$ such that $q_n \to q$
    with $|q| < 1$. Then $g(q) = 0$.
\label{limits:4}
\end{enumerate}
\end{theorem}

\begin{proof}
We prove Parts \eqref{limits:1} and \eqref{limits:2} only.  The proof of the other two
    follows by considering the reciprocal polynomial.

The first follows from an application of Rouch\'e's Theorem.
Let $C := \{z : |z-q| = \varepsilon\}$.
Choose $\varepsilon < |q|-1$ such that the only root of $f$ inside of $C$ is at $q$.
Then there exists an $N$ such that for all $n \geq N$ we have that
    $|f(z) z^n| > |g(z)|$ on $C$.
Hence $f(z) z^n + g(z)$ will have exactly one root inside of $C$ by
    Rouch\'e's Theorem.
As $\varepsilon$ is arbitrary, this proves the result.

For the converse, assume that $q_n \to q$ and $f(q) \neq 0$.
Let $C := \{z : |z-q| = \varepsilon\}$.
Choose $\varepsilon$ such that $f$ has no roots inside
    of $C$, and further that $\varepsilon < |q|-1$.
Then again, there exists an $N$ such that for all $n \geq N$ we have that
    $|f(z) z^n| > |g(z)|$ on $C$.
Hence $f(z) z^n + g(z)$ will have no roots inside of $C$ by
    Rouch\'e's Theorem for all $n \geq N$.
Hence $q_n \not\to q$, a contradiction.
\end{proof}

\begin{theorem}
\label{thm:4.2}
Let $\M(P)$ be the minimal absolute value of the roots of $P$.
\begin{enumerate}
\item
\label{part:1}
We have that
\begin{align*}
1 & = \lim_{r\to\infty} \M(\Phi_r)  \\
  & = \lim_{r\to\infty} \M(\Psi_r)   \\
  & = \lim_{n\to\infty} \M(\X_{A,n}^{\pm} )  \\
  & = \lim_{n\to\infty} \M(\X_{B,n}^{\pm} )
\end{align*}
\item
\label{part:2}
For fixed $r$ we have
\begin{align*}
1 & = \lim_{n\to\infty} \M(\Phi_{C,r,n}^{\pm}) \\
  & = \lim_{n\to\infty} \M(\Psi_{A,r,n}^{\pm}) \\
  & = \lim_{n\to\infty} \M(\Psi_{B,r,n}^{\pm}) \\
  & = \lim_{n\to\infty} \M(\Psi_{C,r,n}^{\pm})
\end{align*}
\item
\label{part:3}
For fixed $r$ we have
\begin{align*}
\lim_{n\to\infty} \M(\Phi_{A,r,n}^{\pm})  & = \kappa_{A,r}  < 1 \\
\lim_{n\to\infty} \M(\Phi_{B,r,n}^{\pm})  & = \kappa_{B,r}  < 1
\end{align*}
\end{enumerate}
\end{theorem}

\begin{proof}
To see Part \eqref{part:1}, we notice that $\Phi_n$, $\Psi_n$ (after multiplying by $(x-1)$)
    $\X_{A, n}^{\pm}$ and $\X_{B, n}^{\pm}$  are
    of the form $f(x) x^n + g(x)$ where $g$ has no roots inside the
    unit circle.
Hence we have that the non-Pisot conjugates of $\Phi_n$ and $\Psi_n$
    tend to 1 in absolute value as $n$ tends to infinity.

To see Part \eqref{part:2}
    we further notice for fixed $r$ that $\Phi_{C,r,n}^{\pm}, \Psi_{A,r,n}^{\pm},
        \Psi_{B,r,n}^{\pm}, \X_{A,n}^{\pm}$ and
        $\X_{B,n}^{\pm}$ all have the
    property that as $n \to \infty$ that all of the non-Pisot conjugates
    tend to $1$ in absolute value.

Lastly, to prove Part \eqref{part:3}, fix $r$.
We notice that $\Phi_{A,r,n}^{\pm}$ will have conjugates
    $q'_n \to q'$ where $q'$ is a root of $x^r - x^{r-1} +1$, with $|q'| < 1$.
We define $\kappa_{A, r}$ to be the minimal absolute value of the roots of $x^r-x^{r-1}+1$.
That is, $\Phi_{A,r,n}^{\pm}$ will have a root tending to $\kappa_{A,r}$ is absolute value as $n \to \infty$.
We have that $\kappa_{A,r}$ is minimized at $r = 3$ with $\kappa_{A,3} \approx 0.754877$.
As similar result holds for $\Phi_{B,r,n}^\pm$, defining $\kappa_{B,r}$ similary and
    with $\kappa_{B,r}$ minimized at $r = 5$.
Here $\kappa_{B,5} \approx 0.84219023$.
\end{proof}

We observe from the above proof that
   for any fixed $r$ that the roots of
   $\Psi_{A,r,n}^{\pm}$ approach the $r+1$-th roots of unity as
   $n \to \infty$.
Moreover, as all of the roots of $x^{r+1}-1$ are of absolute value $1$ we
    see that as $n \to \infty$ we have that $\min |q'| \to 1$ where the
    minimum is taken over the conjugates of $\Psi_{A,r,n}^{\pm}$.
A similar thing can be said for $\Phi_{C,r,n}^{\pm}$ and
    $\Psi_{B, r, n}^{\pm}$.

This doesn't tell the whole tale though.
For example, consider $\Psi_{B, 1, n}^{+} (x) = (x^2-x-1) x^n + 1$.
This clearly has $n+1$ roots inside the unit circle.
Computationally it appears that the roots of this polynomial approach the
    roots of $x^{n+1}-1$.
We will show a slightly weaker result.

\begin{theorem}
Let $C_n = \mathrm{hull} \{q' : |q'| < 1, q' \text{ a root of }
       \Psi_{B, 1, n}^{+} \}$.
For all $\varepsilon$ there exists an $N$ such that for all $n > N$ we have
\[
\{z : |z| \leq 1-\varepsilon \}
  \subset C_n
  \subset \{z : |z| \leq 1 \} \]
\end{theorem}

\begin{remark}
This implies that, given any region
    $\{z: 1-\varepsilon < |z| < 1 \text{ and } \arg(z)
    \in [\theta_1,\theta_2]\}$
    that there will exists some $N$ such that for all $n \geq N$,
    we have $\Psi_{B, q, n}^{+}$ will have a root in
    this region.
\end{remark}

\begin{proof}
The second inclusion is obvious, hence we need only prove the first.

Consider $P(x) := \Psi_{B, 1, n}^{+}(x)^* = x^{n+2} -x^2-x+1$.
Consider the second derivative of $P$, which is
    $P''(x) = (n+2)(n+1) x^n - 2$.
This has roots at $\left(\frac{2}{(n+1)(n+2)}\right)^{1/n} \zeta_n$ where
    $\zeta_n$ is an $n$th root of unity.
We see that $\left(\frac{2}{(n+1)(n+2)}\right)^{1/n} \to 1$ as $n\to \infty$.
Hence as $n$ tends to infinity we have that the roots of $P''$ tend
    uniformly to the unit circle.
Moreover, the convex hull of the roots of $P''$ forms a regular $n$-gon
    with vertices approaching the unit circle.
We recall from the Gauss-Lucas Theorem
    that the complex roots of $P'$ lie within the convex hull of the
    complex roots of $P$.
Further for any $\varepsilon > 0$ we have that for $n$ sufficiently large that
    the roots of $P$ are bounded above in absolute value by $1+\varepsilon$.
As the convex hull of the roots of $P$ contains the roots of $P''$
    which is a regular $n$-gon with vertices arbitrarily close to the unit
    circle, we have our result.
\end{proof}

\begin{cor}
The set of limit points for the conjugates of Pisot numbers contains
    the unit circle.
\end{cor}

\begin{proof}
This follows from Theorem~\ref{thm:4.2} and \cite[Theorem 1]{ErdosTuran50}.
\end{proof}

In fact, this is already true for the multinacci numbers, i.e., the roots of $x^n-x^{n-1}-\dots -x-1$.

\section{Pisot numbers in $(m,m+1)$}
\label{sec:m}

Recall that Boyd \cite{Boyd84,Boyd85} has given an algorithm that finds all Pisot
    numbers in an interval, where, in the case of limit points, the algorithm
    can detect the limit points and compensate for them.
Using this algorithm, combined with the results above
    for regular Pisot numbers, we can computationally show

\begin{theorem}
\label{thm:1.933}
All Pisot numbers $q \in (1, 1.933]$ satisfy Conjecture~\ref{conj:exact}.
\end{theorem}

\begin{proof}
Using the algorithms above, we note that there are 760 irregular Pisot numbers
    less than 1.933.
These can be found at \cite{Hare}.
We explicitly check that they all satisfy Conjecture~\ref{conj:exact}.
We see from Theorem~\ref{thm:m=1} that all regular Pisot numbers
    less than 1.933 satisfy Conjecture~\ref{conj:exact}.
This completes the proof.
\end{proof}

We choose $1.933$ due to the limit point at
    $\phi_4 \approx 1.933184982$, the root in $(1,2)$ of $x^5-2 x^4+x-1$.
Although each limit point can be handled by the algorithm of
    \cite{Boyd84, Boyd85}, the handling of such limit points gets more
    complicated the closer we get to $2$.

Alternately, we can look at a finite (albeit large) set of Pisot numbers
    in $(1,2)$, but with bounded degree.

\begin{theorem}
\label{thm:150}
All Pisot numbers $q \in (1, 2)$ of degree at most 150
    satisfy Conjecture~\ref{conj:exact}.
\end{theorem}

We noticed that Pisot numbers in $(1, 1/c_1)$ necessarly satisfy
    Conjecture~\ref{conj:exact}.
There are 59876 Pisot numbers in $[1/c_1, 2)$ of degree at most 150,
    so this is strong heuristic evidence.

For each $m$ and $N$ in Table~\ref{tab:counter} we have
   listed the number of polynomials with a unit Pisot root
   in $(1/c_m, m+1)$ of degree at most $N$.
We note that if the Pisot number is in $(m, 1/c_m)$ or a non-unit Pisot,
   then we necessarily have that all conjugates are greater than $c_m$.
A complete list of these polynomials is provided at
    \cite{Hare}.
We also have listed the Pisot polynomials with the smallest
    known conjugate in modulus.
Based on this table, we make Conjecture~\ref{conj:exact}.

\begin{table}[H]
\begin{tabular}{llll}
$m$ & $N$ & Size & smallest polynomial \\
\hline
1  & 150& 59876  & $x^2-x-1$  \\
2  & 12 & 50557 & $x^4-3 x^3+x^2-2 x-1$  \\
3  & 9  & 67213  & $x^2-4x+1$ \\
4  & 7  & 18995  & $x^2-5x+1$ \\
5  & 7  & 50317  & $x^2-6x+1$ \\
6  & 6  & 15268  & $x^2-7x+1$ \\
7  & 6  & 26959  & $x^2-8x+1$ \\
8  & 5  & 4696   & $x^2-9x+1$ \\
9  & 5  & 6578   & $x^2-10x+1$ \\
10 & 5  & 8743   & $x^2-11x+1$ \\
\end{tabular}
\caption{Pisot polynonimals with smallest conjugate in modulus}
\label{tab:counter}
\end{table}

\section{Garsia Entropy}
\label{sec:Garsia}

Given $\beta\in(1,2)$, the Bernoulli convolution $\nu_{\beta}$ is the
weak$^*$ limit of the measures $\nu_{\beta,n}$ given by
\[
\nu_{\beta,n}=\sum_{a_1 \ldots a_n\in\{0,1\}^n}
\frac{1}{2^n}\delta_{\sum_{i=1}^n a_i\beta^{-i}}.
\]
These interval supported self-similar measures have been intensely studied
    since the 1930s.
The Bernoulli convolutions are known to be exact-dimensional \cite{FengHu09}.
That is
\[
\dim(\nu_\beta)=\lim_{h\to0}\frac{\log\nu_\beta(x,x+h)}{\log h}
\]
for $\nu_\beta$-a.e. $x$.

In particular, the dimension of $\nu_\beta$ equals its Hausdorff dimension.
In particular, the question of which parameters $\beta$ give rise to
measures which are singular or which have dimension less than one has
been extremely well studied. Erd\H{o}s \cite{Erdos39}
showed that Pisot numbers give
rise to singular Bernoulli convolutions, and Garsia \cite{Garsia63} showed
that $\dim(\nu_{\beta})<1$ when $\beta$ is Pisot. It remains unknown
whether there are any other parameters that give rise to singular
Bernoulli convolutions.

Recent work of Hochman \cite{Hochman} showed that for algebraic $\beta$ the dimension
of $\nu_{\beta}$ can be given explicitly in terms of the Garsia
entropy of $\beta$,
\begin{equation}\label{GarsiaDim}
  \dim(\nu_{\beta})=\min\left\{1,
  \frac{H(\beta)}{\log(\beta)}\right\}.
\end{equation}
The Garsia entropy $H(\beta)$, defined in \cite{Garsia63}, is
a quantity which measures how often different words $a_1 \ldots a_n$
give rise to the same sum $\sum_{i=1}^n a_i\beta^{-i}$.  It should be
noted that the Garsia entropy is sometimes defined after normalizing
by $\log(\beta)$.
Recall, we say $\beta$ is a height $h$ algebraic number if it is the
    root of an integer polynomial whose coefficents are bounded by
    $h$ in absolute value.
It follows from the definition that if $\beta$ is not a
height~1 algebraic integer, then $H(\beta) = \log(2)$, and
by \eqref{GarsiaDim}, $\dim(\nu_\beta) = 1$.
Very recently Varj\'u, \cite{Varju}, has shown that if $\beta$ is
non-algebraic, then $\dim(\nu_{\beta})=1.$ Thus the problem of
understanding which parameters $\beta$ give rise to Bernoulli
convolutions of dimension less than one has been reduced to one of
understanding Garsia entropy for algebraic $\beta$ of height~1.
For more information on Bernoulli convolutions, see \cite{VarjuA}.

In \cite{HKPS} it is shown that if $\beta_1 > 1$ is an algebraic number with
    real conjugate $\beta_2$ such that
   \[ \log(\beta_1)/ |\log(\beta_2)| < 0.82, \]
then the Garsia entropy of the Bernoulli convolutions of $\beta_1$ is strictly
    greater than 1, and hence the dimension of this Bernoulli
    convolution is exactly 1.
In particular, if $q \in (1,2)$ is a (non-quadratic)
    Pisot number and $q'$ is a real
    conjugate, then this shows that $\beta_1 = 1/|q'|$ satisfies these
    conditions if $|q'| > q^{-0.82}$.

\begin{conj}[Conjecture 3.5 of \cite{HKPS}] \label{conj:3.5}
If $\beta_1 \in (1,2)$, $\beta_1 \neq \frac{1+\sqrt{5}}{2}$,
     has a real conjugate $\beta_2$ such that $1/|\beta_2|$
    is a Pisot number, then $\dim(\nu_{\beta_1}) = 1$.
\end{conj}

\begin{remark}
Numbers $\beta_1$, as described in Conjecture \ref{conj:3.5}, are
    are often called {\em anti-Pisot}, see, e.g., \cite{SiSo11}.
\end{remark}

\begin{theorem}
\label{thm:exact}
Conjecture~\ref{conj:exact} implies Conjecture~\ref{conj:3.5}.
\end{theorem}

\begin{remark}
What is proven here is slightly stronger than this.
We show that
\begin{itemize}
\item If $\beta_1 \in (1,2)$ has a conjugate $\beta_2$ such that
     $1/|\beta_2| \in (1, 1.8)$ is a Pisot number, then $\dim(\nu_{\beta_1}) = 1$.
\item If $\beta_1 \in (1,2)$ has a conjugate $\beta_2$ such that
     $1/|\beta_2| \in (1, 2)$, then Conjecture~\ref{conj:exact} implies $\beta_1 \in (1, c_1]$.
\item If $\beta_1 \in (1,c_1)$ has a conjugate $\beta_2$ such that
     $1/|\beta_2| \in (1.8, 2)$ is a Pisot number, then $\dim(\nu_{\beta_1}) = 1$.
\end{itemize}
\end{remark}

\begin{proof}
Let $\beta_1$ be an anti-Pisot number with conjugate $\beta_2$.
The plan is to show that
\begin{enumerate}
\item If $1/|\beta_2|$ is an irregular Pisot numbers less than 1.8, then $\dim(\nu_{\beta_1})=1$.
\label{check:1}
\item If $1/|\beta_2| \in [1.8, 2)$, then Conjecture~\ref{conj:exact} implies that $\dim(\nu_{\beta_1}) = 1$.
\label{check:3}
\item If $1/|\beta_2|$ is a regular Pisot number less than 1.8, then $\dim(\nu_{\beta_1}) = 1$.
\label{check:2}
\end{enumerate}

Part \eqref{check:1} was in fact verfied for the 760 irregular Pisot numbers
    less than 1.933, as these irregular Pisot numbers were previously computed for
    the proof of Theorem~\ref{thm:1.933}.

Part \eqref{check:3} follows by noticing that
    $\frac{\log\left(\frac{1+\sqrt{5}}{2}\right)}{\log (1.8)} \approx 0.818$
and hence Conjecture~\ref{conj:3.5} will hold for all
    for $|\beta_2| > c_1$ and all $\beta_1 > 1.8$
    subject to Conjecture~\ref{conj:exact}.

To prove Part \eqref{check:2}, we need to study the regular Pisot numbers
    less than 1.8.
There are only two limit points of Pisot numbers less than $1.8$.
    namely $\phi_1 = \psi_1$ and $\phi_2$.
Hence, the regular Pisot numbers less than $1.8$ are of one of the
    following forms:
\begin{enumerate}
\item Either $\phi_1=\psi_1$ or $\phi_2$.
\label{pisot:1}
\item
In a regular family as described in Table~\ref{tab:regular} approaching
    one of the two limit points less than 1.8.
That is, the Pisot root less than 1.8
    of one of $\Phi_{A,1,n}^{\pm}, \Phi_{B,1,n}^{\pm}, \Phi_{C,1,n}^{\pm}, \Psi_{A,1,n}^{\pm}, \Psi_{B,1,n}^{\pm},
                          \Phi_{A,2,n}^{\pm}, \Phi_{B,2,n}^{\pm}$ or $\Phi_{C,2,n}^{\pm}$.
\label{pisot:2}
\item A regular Pisot number less than $1.8$ but part of a family
    approaching a limit point greater than $1.8$.
That is, the Pisot root less than 1.8 of one of
$\Phi_{A, r, n}^\pm,
\Phi_{B, r, n}^\pm,
\Phi_{C, r, n}^\pm$ for $r \geq 3$,
$\Phi_{A, r, n}^\pm,
\Phi_{B, r, n}^\pm$ for $r \geq 2$, or
$\X_{A, n}^\pm$ or $\X_{B, n}^\pm$.
\label{pisot:4}
\end{enumerate}

Case \eqref{pisot:1} is easy,
    and we just explicitly check that these both have the
    desired property.

For Case \eqref{pisot:2} we
    can use an explicit version of Theorem~\ref{thm:rouche} Part~\eqref{limits:3}
    to find a $N$ such that for all $n \geq N$ we have the roots of
    $f(x) x^n + g(x)$ satisfy the desired property.
We will do the case of $\Phi_{B, 1, n}^{+}$ only,
    as the rest are similar.

We see that $\Phi_{B,1,n}^{+}(x) = (x^2-x-1) x^n  + 1$.
This has a limit of $\phi_1 \approx 1.618$.
We notice that $1.6^2 -1.6 -1 = -0.04$ and $1.64^2 - 1.64^2 -1 =
    0.0496$.
We have for $n \geq 7$ that $1.6^n (1.6^2-1.6-1) + 1< 0$ and
                            $1.64^n (1.64^2-1.64-1) + 1 > 2$.

Hence for $n \geq 7$ we have the Pisot root of $\Phi_{B,1,n}^+$
    is in $[1.6, 1.64]$.

Consider the circle $C = \{z: |z| = 0.7 \}$.
Let $f_n(x) = x^n (x^2-x-1)$ and $g(x) = 1$.
Evaluating $f_n(x) = (x^2-x-1)x^n$ on this circle, we see that
    $|(x^2-x-1) x^n| \leq |0.7^2 + 0.7 + 1| 0.7^n = 2.19 \cdot 0.7^n$.
We see for $n \geq 3$ that $|x^n (x^2-x-1)| < 1 = g(x)$ on $C$.
As the polynomial $g$ has no roots, it has no roots inside $C$.
Hence $f_n(x) + g(x) = x^n(x^2-x-1) + 1$ has no roots inside of $C$ for
    $n \geq 3$.

Combining these together, we see that the Pisot root of $\Phi_{B,1,n}^+$
    is greater than 1.6, and all of its conjugates are greater than 0.7.
This implies that $\frac{|\log q'|}{\log(q)} \leq \frac{|\log(0.7)|}{\log(1.6)} \approx 0.7588$.
Hence, we have that for all $n \geq 7$ that the roots of $\Phi_{B,q,n}$ satisfy
    Conjecture~\ref{conj:3.5}.

Consider Case \eqref{pisot:4}.
Although there are a multiple infinite families of regular Pisot numbers
    corresponding to the limit points less than 1.8,
    there are only 5 regular Pisot numbers
    less than 1.8 that are not part of one of these infinite families.
\begin{table}
\begin{tabular}{ll}
Minimal polynomial & root \\
\hline
$x^6-x^5-x^4-x-1$& 1.743700166 \\
$x^7-x^6-x^5-x^3-1$& 1.774520059 \\
$x^7-2 x^6+2 x^4-2 x^3+x-1$& 1.683468801 \\
$x^7-x^6-x^5-x^4+x^3-1$& 1.747457424 \\
$x^7-2 x^6+x^5-2 x^4+2 x^3-x^2+x-1$& 1.790222867
\end{tabular}
\caption{Exceptional regular Pisot numbers}
\label{tab:reg exception}
\end{table}
Only one of these has a real root less than 1, and all of them
    satisfy Conjecture~\ref{conj:3.5}.
\end{proof}

\section{Conclusions and open questions}

The case for which Conjecture~\ref{conj:exact} remains unproven for $q\in(1,2)$ is $q$ being an irregular Pisot number in $(1.933,2)$ of degree at least 151 and not a simple Parry number. As mentioned above, it is not known whether there are finitely many such numbers.

We have shown that there are convergent sequences of regular Pisot numbers for which the sequence of their conjugates does not tend to 1 in modulus. This is an interesting phenomenon which requires further investigation. Most of these conjugates will tend to the unit circle and their arguments are distributed in a nice way, see \cite{Bilu88, Bilu97, ErdosTuran50}.

\end{document}